\documentclass[12pt]{amsart}
\usepackage[margin = 1.0in]{geometry}

\usepackage{tikz}
\usepackage{tikz-cd}
\usepackage{url, hyperref}
\usetikzlibrary{shapes.arrows}
\usetikzlibrary{decorations.pathreplacing}
\usepackage[all]{xy}

\usepackage{amsmath}
\usepackage{amssymb}
\usepackage{enumitem}
\usepackage{graphicx}
\usepackage{mathdots}
\usepackage{color}
\usepackage{diagbox}
\usepackage{array, makecell}
\usepackage{rotating}
\usepackage{amsmath}
\usepackage{tikz-cd}
\usepackage{ytableau}

\def\bC{\mathbb{C}}
\def\bP{\mathbb{P}}

\def\cL{\mathcal{L}}

\def\cT{\mathcal{T}}

\def\bL{\mathbf{L}}

\DeclareMathOperator{\Gr}{Gr}
\DeclareMathOperator{\Tev}{Tev}

\DeclareMathOperator{\SYT}{SYT}

\usepackage{amsthm}
\newtheorem{definition}{Definition}[section]
\newtheorem{theorem}[definition]{Theorem}
\newtheorem{example}[definition]{Example}
\newtheorem{proposition}[definition]{Proposition}

\newtheorem{lemma}[definition]{Lemma}
\newtheorem{remark}[definition]{Remark}

\title{Curves in projective space and RSK}

  \author{Carl Lian}
\address{Washington University in St. Louis, Department of Mathematics, 1 Brookings Drive
\hfill \newline\texttt{}
 \indent  St. Louis, MO 63130} \email{{\tt clian@wustl.edu}}

\author{Saskia Solotko}
\address{Tufts University, Department of Mathematics, 177 College Ave
\hfill \newline\texttt{} 
 \indent Medford, MA 02155} \email{{\tt Saskia.Solotko@tufts.edu}}

 \date{\today}

\usepackage{graphicx}

\begin{document}

\begin{abstract}
    The geometric Tevelev degrees of projective space enumerate general, pointed algebraic curves interpolating through the maximal possible number of points. Previous work expresses these invariants in terms of Schubert calculus. Extending ideas of Gillespie--Reimer-Berg, we use the RSK correspondence to give a positive interpretation of these counts in terms of the combinatorics of words.
\end{abstract}

\maketitle

\date{\today}

\setcounter{tocdepth}{1}

\tableofcontents
\section{Introduction}\label{sec:intro}

Let $C$ be a general smooth, projective algebraic curve of genus $g$, defined over the field of complex numbers $\bC$. Let $\bP^r$ be the $r$-dimensional complex projective space. Let $p_1,\ldots,p_n\in C$ and $x_1,\ldots,x_n\in \bP^r$ be fixed, general (hence distinct) points. Assume that
\begin{equation}\label{eq:dim_constraint}
    n=\frac{r+1}{r}\cdot d-g+1\ge r+1.
\end{equation}
Then, the \emph{geometric Tevelev degrees} $\Tev^{\bP^r}_{g,n,d}$ enumerate degree $d$ maps $f:C\to\bP^r$ satisfying $f(p_i)=x_i$ for each $i=1,\ldots,n$. That is, $\Tev^{\bP^r}_{g,n,d}$ is by definition the number of copies of a general pointed curve in $r$-dimensional space interpolating through the maximal number of points.

The purpose of this paper is to give a combinatorial proof of the following theorem.

\begin{theorem}\label{thm:main_intro}
Assume \eqref{eq:dim_constraint}. Then, $\Tev^{\bP^r}_{g,n,d}$ is equal to the number of $(r+1)$-ary words of length $g$ (Definition \ref{def:word}) which have:
\begin{enumerate}
\item a collection of at least $g+r-d$ disjoint decreasing subsequences of length $r+1$ (Definition \ref{def:decreasing})\footnote{Condition (i) is vacuous if $d\ge g+r$.},
\item no nondecreasing subsequence of length greater than $\frac{d}{r}$ (Definition \ref{def:decreasing}), and
\item no $(i,i+1)$-subsequence of length $\frac{d}{r}.$ (Definition \ref{def:(i,i+1)_subseq})
\end{enumerate}
\end{theorem}

The values of $\Tev^{\bP^r}_{g,n,d}$ have previously been determined algebro-geometrically in \cite{tev_Pr} in terms of Schubert calculus, see Theorem \ref{thm:tev_Pr} below. Our main contribution is to translate the answer into that of Theorem \ref{thm:main_intro} using the RSK correspondence, building on earlier work of Gillespie--Reimer-Berg \cite{GenRSK}. In the remainder of the introduction, we review the previous results and discuss what is new in more detail.

\subsection{Geometric results}

We give an overview of the geometric problem of computing $\Tev^{\bP^r}_{g,n,d}$. The discussion is mainly for context, and does not play an essential role in the arguments appearing in the main body of the paper.

\begin{theorem}\cite[Theorem 1.1, Theorem 7.9]{tev_Pr}\label{thm:tev_Pr}
Under the assumption \eqref{eq:dim_constraint}, we have
    \begin{align*}
\Tev^{\bP^r}_{g,n,d}&=\displaystyle\int_{\Gr(r+1,d+1)}\sigma^g_{1^r}\cdot \left(\sum_{\mu\subset(n-r-2)^{r}}\sigma_{\mu}\sigma_{\overline{\mu}}\right)_{\lambda_0\le n-r-1}\\
&=\displaystyle\int_{\Gr(r+1,d+1)}\sigma^g_{1^r}\cdot \left(\sum_{\lambda\subset (n-r-1)^{r+1}} \Gamma^\lambda_{r,n,n-1}\sigma_{\lambda}\right)
\end{align*}
\end{theorem}

The requirement \eqref{eq:dim_constraint} that $n\ge r+1$ ensures that the interpolating maps $f:C\to\bP^r$ in question are non-degenerate, that is, their images do not lie in any hyperplane. Brill-Noether theory implies that a general pointed curve of degree $d$ and genus $g$ can interpolate through no more than $\frac{r+1}{r}\cdot d-g+1$ general points in $\bP^r$, and that when $n$ is equal to exactly this number, the set of such interpolating maps $f$ is finite and transverse.\footnote{That is, the interpolating maps $f$ admit no infinitesimal deformations fixing the points $p_i,x_i$.}

See also \cite{tevelev,cps,fl,cl1,Troptev} for work on the special case $r=1$, and \cite{lian_hyp,cl2,cl_hirzebruch,cl_complete,ls} for work on geometric Tevelev degrees of other target varieties. As alluded to below, the question of \emph{virtually} enumerating fixed general curves on algebraic varieties goes back to the beginning of the subject of Gromov-Witten theory, though we point out that the Theorem \ref{thm:tev_Pr} gives an \emph{actual} count of curves, as opposed to a virtual intersection number.

Both formulas of Theorem \ref{thm:tev_Pr} are intersection numbers on the Grassmannian manifold $\Gr(r+1,d+1)$, and the elements $\sigma_{\lambda}$, for partitions $\lambda$, are Schubert classes lying in its cohomology ring. The integers $\Gamma^{\lambda}_{r,n,n-1}\ge0$ enumerate \emph{1-strip-less} semi-standard Young tableaux, see \cite{lian_lr} and Definition \ref{ii+1stripdef} below. The expressions
\begin{equation*}
    \sum_{\mu\subset(n-r-2)^{r}}\sigma_{\mu}\sigma_{\overline{\mu}}=\sum_{\lambda\subset (n-r-1)^{r+1}} \Gamma^\lambda_{r,n,n-1}\sigma_{\lambda}.
\end{equation*}
arise as the classes of generic torus orbit closures on the Grassmannian $\Gr(r+1,n)$, going back to work of Klyachko \cite{klyachko}, Anderson-Tymoczko \cite{at}, and Berget-Fink \cite{bf}. The two formulas are compared via the Littlewood-Richardson rule in \cite{lian_lr}. We do not review the notation for Schubert classes here, referring instead to \cite{tev_Pr}. In this paper, we will pass instead to the natural interpretation of the intersection numbers as enumerations of \emph{$L$-tableaux}, see \cite[Definition 1.1]{GenRSK} and Theorem \ref{carlthm} below.

For curves of large degree, it had been observed earlier that the geometric Tevelev degrees $\Tev^{\bP^r}_{g,n,d}$ admit simpler formulas.

\begin{theorem}\cite[Theorem 1.1, Theorem 1.2]{fl}\label{thm:large_deg}
Assume \eqref{eq:dim_constraint}, and assume further that $d\ge rg+r$ (equivalently, that $n\ge rg+r+2$ or that $n\ge d+2$). Then, 
    \begin{align*}
\Tev^{\bP^r}_{g,n,d}&=\displaystyle\int_{\Gr(r+1,d+1)}\sigma^g_{1^r}\cdot \left[\sum_{\alpha_0+\cdots+\alpha_r=(r+1)(d-r)-rg}\left(\prod_{j=0}^{r}\sigma_{\alpha_j}\right)\right]\\
&=(r+1)^g.
\end{align*}
\end{theorem}
When $d\ge rg+r$, the equality
\begin{equation*}
\sum_{\alpha_0+\cdots+\alpha_r=(r+1)(d-r)-rg}\left(\prod_{j=0}^{r}\sigma_{\alpha_j}\right)=\sum_{\lambda\subset (n-r-1)^{r+1}} \Gamma^\lambda_{r,n,n-1}\sigma_{\lambda}
\end{equation*}
of expressions appearing in Theorems \ref{thm:tev_Pr} and \ref{thm:large_deg} follows from a straightforward application of the Pieri rule. Indeed, the 1-strip-less condition is empty in this range. The answer $(r+1)^g$ matches the corresponding \emph{virtual} enumeration of curves in Gromov-Witten theory \cite[(3)]{bp}, and in fact goes back to early work on Vafa-Intriligator formulas \cite{bdw,st,mo}. The virtual invariants are enumerative when $d\ge rg+r$, but not otherwise, see \cite{Enumvir,bllrst,celadoan} for further explorations of this phenomenon for different targets.

\subsection{Combinatorial results}

The Schubert calculus and $(r+1)^g$ formulas of Theorem \ref{thm:large_deg} are obtained in \cite{fl} by two different geometric calculations, leaving open the question of whether there is a simple combinatorial explanation for their equality. Such an explanation was provided in work of Gillespie--Reimer-Berg \cite{GenRSK}, which is the starting point for this paper.

\begin{theorem}\cite[Theorem 1.3]{GenRSK}\label{thm:grb}
    Suppose that $d\ge g+r$. Then, via the RSK correspondence,
    \begin{equation*}
\displaystyle\int_{\Gr(r+1,d+1)}\sigma^g_{1^r}\cdot \left[\sum_{\alpha_0+\cdots+\alpha_r=(r+1)(d-r)-rg}\left(\prod_{j=0}^{r}\sigma_{\alpha_j}\right)\right]=(r+1)^g.
\end{equation*}
\end{theorem}
More precisely, after suitable transformations, the intersection number
\begin{equation}\label{fl_schubert}
    \int_{\Gr(r+1,d+1)}\sigma^g_{1^r}\cdot \left[\sum_{\alpha_0+\cdots+\alpha_r=(r+1)(d-r)-rg}\left(\prod_{j=0}^{r}\sigma_{\alpha_j}\right)\right]
\end{equation}
is interpreted as the number of pairs $(P,Q)$ of a semi-standard and standard tableau of the same shape, of size $g$, and of height at most $(r+1)$. The RSK algorithm converts such pairs into $(r+1)$-ary words of length $g$, of which there are $(r+1)^g$. Note that the result of Theorem \ref{thm:grb} is valid under the assumption $d\ge r+g$, weaker than that of Theorem \ref{thm:large_deg}, but \eqref{fl_schubert} is no longer equal to $\Tev^{\bP^r}_{g,r,d}$ in the wider range.

The more general formula of Theorem \ref{thm:tev_Pr} now being available, it is natural to extend Gillespie--Reimer-Berg's RSK calculation to the full range of geometric Tevelev degrees $\Tev^{\bP^r}_{g,n,d}$. It can no longer be true that the maps enumerated by $\Tev^{\bP^r}_{g,n,d}$ are in bijection with $(r+1)$-ary words of length $g$. Indeed, the requirement that $d\ge rg+r$ for the equality $\Tev^{\bP^r}_{g,n,d}=(r+1)^g$ is essentially sharp, see e.g. \cite[Example 1.6]{tev_Pr}. 

On the other hand, it is true in general that $\Tev^{\bP^r}_{g,n,d}\le (r+1)^g$ \cite[Corollary 5.8]{tev_Pr}, so RSK could still afford an \emph{injection} into the set of $(r+1)$-ary words of length $g$ with combinatorially meaningful image. Our main result, Theorem \ref{thm:main_intro}, shows that this is the case. Theorem \ref{thm:main_intro} moreover recovers the fact that $\Tev^{\bP^r}_{g,n,d}=(r+1)^g$ when $d\ge rg+r$. Indeed, condition (i) of Theorem \ref{thm:main_intro} is vacuous if $d\ge g+r$, and conditions (ii) and (iii) are vacuous if $\frac{d}{r}>g$. Further specializations of Theorem \ref{thm:main_intro} are discussed in \S\ref{sec:remarks}.

Our proof of Theorem \ref{thm:main_intro} begins along the lines of the work of Gillespie--Reimer-Berg \cite{GenRSK}, converting $L$-tableaux enumerated by the formula of Theorem \ref{thm:tev_Pr} into pairs of tableaux $(P,Q)$ of the same shape. The previous work dealt only with the case $d\ge r+g$; a modification is needed in the range $d<r+g$.\footnote{$d\ge r+g$ and $d<r+g$ are the ``Brill-Noether-general'' and ``Brill-Noether-special'' ranges, respectively. That is, if $C$ is a general curve of genus $g$, then every line bundle $\cL$ of degree $d$ has $\dim H^0(C,\cL)\ge r+1$ (and hence underlies a non-degenerate map $f:C\to\bP^r$, possibly after twisting base points) if and only if $d\ge r+g$. While Brill-Noether theory plays a central role in the geometric calculation of $\Tev^{\bP^r}_{g,n,d}$, we do not have a satisfying explanation for the appearance of this dichotomy in the tableaux combinatorics.}

The new phenomenon is that restrictions arise on the tableaux $(P,Q)$. Theorem \ref{thm:main_intro} then amounts to determining the RSK-image of the subset of restricted pairs $(P,Q)$. Conditions (i) and (ii) of Theorem \ref{thm:main_intro} arise from Greene's Theorem (\cite[Theorem 4.8.10]{sagan} and \cite{greene}), relating the shapes of $P,Q$ to increasing and decreasing subsequences of the RSK-image. Condition (iii) arises from the 1-strip-less condition of \cite{lian_lr} on the semi-standard tableau $P$ and an analysis of Knuth equivalence.

We recall Gillespie--Reimer-Berg's L-tableaux and the enumeration of the geometric Tevelev degrees $\Tev^{\bP^r}_{g,n,d}$ in terms of L-tableaux in \S\ref{sec:prelim}. In \S\ref{sec:tableaux_operations}, we convert the L-tableaux into pairs of tableaux of the same shape, extending the ideas of \cite{GenRSK}. We compute the corresponding $(r+1)$-ary words under RSK in \S\ref{sec:rsk}, thereby proving Theorem \ref{thm:main_intro}.

We assume some familiarity with Young tableaux and the RSK correspondence, see for instance \cite[\S 7]{stanley}. We write the parts of partitions $\lambda=(\lambda_1,\ldots,\lambda_\ell)$ in non-increasing order. We take the convention that Young diagrams are top- and left-justified. A filling of the Young diagram of $\lambda$ is \emph{semi-standard} if entries are strictly increasing down columns from top to bottom, and weakly increasing across rows from left to right.

\subsection{Acknowledgments}

Both authors were supported by the Summer Scholars research program at Tufts University.  C.L. has been supported by NSF Postdoctoral Fellowship DMS-2001976, the MATH+ incubator grant “Tevelev degrees,” and an AMS-Simons travel grant. We also thank Maria Gillespie and Andrew Reimer-Berg for explaining their results to us.

\section{Preliminaries}\label{sec:prelim}

We begin by defining the fillings of rectangles introduced in \cite[Propositon 4.1]{fl}, called \emph{L-tableaux} in \cite{GenRSK}. 

\begin{definition}\label{Ltableaux}
An \emph{$L$-tableau} $\bL$ with parameters $(g,r,d)$ is a filling of boxes of an $(r+1)\times (d-r)$ grid with
\begin{itemize}
    \item $(d-r)(r+1)-rg$ \color{blue} blue \color{black} integers among $\textcolor{blue}{1,2,\ldots, r+1}$, with each appearing any number of times. The blue integers are top- and left- justified, strictly increasing down columns and weakly increasing across rows from left to right. The blue integers therefore form a semi-standard Young tableau (SSYT), which we call the \emph{blue tableau} $B$.

    \item $rg$ \color{red}red \color{black} integers among $\textcolor{red}{1,2,\ldots, g}$, with each appearing exactly $r$ times. The red integers are bottom- and right- justified, weakly decreasing down columns and strictly increasing across rows from right to left. The red integers therefore form a $180^\circ$-rotated and conjugated SSYT which we call the \emph{red tableau} $R$.
\end{itemize}
\end{definition}
Here, we take the conventions of \cite{tev_Pr}. Our $L$-tableaux become the fillings of \cite[Proposition 4.1]{fl} after rotating $180^\circ$ and replacing each blue integer $\textcolor{blue}{i}$ with $\textcolor{blue}{r-i}$. Our $L$-tableaux become those of \cite{GenRSK} after reflecting over the vertical axis and replacing each blue integer $\textcolor{blue}{i}$ with $\textcolor{blue}{r+1-i}$.

\begin{example}\label{eg:L-tab} An $L$-tableau with parameters $(g,r,d)=(11,2,12).$ The blue integers (top- and left- justified) are boldfaced.
\begin{center}
\vspace{.55cm}
\begin{ytableau}

 \textcolor{blue}{\textbf{1}}& \textcolor{blue}{\textbf{1}} & \textcolor{blue}{\textbf{1}} & \textcolor{blue}{\textbf{2}} & \textcolor{blue}{\textbf{3}} & \textcolor{red}{8}& \textcolor{red}{6}& \textcolor{red}{5}& \textcolor{red}{3}& \textcolor{red}{2}\\

 \textcolor{blue}{\textbf{2}}& \textcolor{blue}{\textbf{2}} & \textcolor{red}{11} & \textcolor{red}{10}&\textcolor{red}{9}& \textcolor{red}{7}& \textcolor{red}{6}& \textcolor{red}{4}& \textcolor{red}{3}& \textcolor{red}{1}\\

 \textcolor{blue}{\textbf{3}}& \textcolor{red}{11} & \textcolor{red}{10} & \textcolor{red}{9}&\textcolor{red}{8}& \textcolor{red}{7}& \textcolor{red}{5}& \textcolor{red}{4}& \textcolor{red}{2}& \textcolor{red}{1}

\end{ytableau}
\end{center}

\label{fig:itsanltableau}
\vspace{.3cm}
\end{example}

In \cite{tev_Pr}, additional constraints are imposed on the blue tableaux involving \emph{strips}, which we now define:

\begin{definition}\label{stripdef}
A \emph{strip} $S$ with length $\ell$ of a Young diagram $\lambda$ is a collection of $\ell$ boxes in the Young diagram with the property that:
\begin{itemize}
\item Exactly one box of $S$ lies in each of the first $\ell$ columns of $\lambda,$ and
\item Given any distinct boxes $b_1,b_2$ of $S,$ if $b_1$ lies in a column to the left of $b_2,$ then $b_1$ does not also lie in a row above $b_2.$

\end{itemize} 
\end{definition}

\begin{example}\label{eg:strip} A Young diagram of shape $\lambda=(7,4,4,3)$ with a strip (highlighted green) of length $\ell=6.$

\begin{center}
\vspace{.3cm}
\begin{ytableau}
\text{ } & & & &*(green) & *(green)&  \\
& *(green)& *(green)&*(green)  \\
*(green)& & &  \\
& &  \\
\end{ytableau}
\vspace{.3cm}
\end{center}
\end{example}

% \begin{figure}[H]
% \centering
% \includegraphics[scale=.5]{strip.png}

% \caption{An example of a strip}
% \end{figure}

\begin{definition}\label{ii+1stripdef}

For $i\in \{1,2,\cdots, r\},$ an \emph{$(i,i+1)$-strip} of a SSYT $B$ is a strip of the underlying Young diagram $\lambda$, all of whose boxes are filled with the entry $i$ or $i+1,$ and for which all instances of $i$ all appear to the left of all instances of $i+1.$

\end{definition}

In \cite{lian_lr,tev_Pr}, an $(i,i+1)$-strip for some (unspecified) value of $i$ is called a \emph{1-strip}. An $(i,i+1)$-strip is allowed to be filled entirely with the entry $i$ (resp. $i+1$), in which case it is also an $(i-1,i)$-strip (resp. $(i,i+1)$-strip), as long as $i>1$ (resp. $i<r$).

% In the following, we utilize a fourth parameter $n:=\frac{r+1}{r}d-g+1.$ Then the number of blue integers is $(d-r)(r+1)-rg=r(n-r-2).$

We now recall the interpretation of the geometric Tevelev degrees $\Tev^{\bP^r}_{g,n,d}$ of $\bP^r$ in terms of $L$-tableaux. We adopt the setup of \S\ref{sec:intro}; in particular, we write 
\begin{equation*}
    n=\frac{r+1}{r}\cdot d-g+1,
\end{equation*}
which is assumed throughout to be an integer. The number of blue integers in an $L$-tableaux with parameters $(g,r,d)$ is equal to
\begin{equation*}
    (d-r)(r+1)-rg=r(n-r-2).
\end{equation*}

\begin{theorem}{\cite[Theorem 7.9]{tev_Pr}}
\label{carlthm} The count $\Tev^{\bP^r}_{g,n,d}$ is equal to the number of $L$-tableaux $\bL$ with parameters $(g,r,d)$ satisfying the following additional constraints:

\begin{itemize}
    \item The blue tableau $B$ is supported in the leftmost $n-r-1=\frac{r+1}{r}\cdot d-g-r$ columns of $\bL$.
    \item No $(i,i+1)$-strip of length $n-r-1$ appears in $B$, for any value $i\in \{1,2,\cdots, r\}.$
\end{itemize}
\end{theorem}

\begin{definition}\label{def:calL}
    Let $\cL_{g,n,d}^r$ be the set of $L$-tableaux $\bL$ with parameters $(g,r,d)$ satisfying the additional constraints of Theorem \ref{carlthm}.
\end{definition}

The $L$-tableau appearing in Example \ref{eg:L-tab} lies in $\cL_{11,8,12}^2$.

\begin{remark}
We have only required in \eqref{eq:dim_constraint} that $n\ge r+1$, whereas the number of blue integers required to be in $\bL$ is only non-negative if $n\ge r+2$. Accordingly, if $n=r+1$, then $\Tev^{\bP^r}_{g,n,d}=0$. Indeed, in this case, the Brill-Noether theorem implies that no non-degenerate maps $f:C\to\bP^r$ of degree $d$ exist.
\end{remark}

\begin{remark}\label{rem:conditions_vacuous}
% The first condition on $B$ is vacuous if
% \begin{equation*}
%     n-r-1\ge d-r \Leftrightarrow d\ge rg.
% \end{equation*}
% If in addition
% \begin{equation*}
%     n-r-1> d-r \Leftrightarrow d\ge rg+r,
% \end{equation*}
% then both conditions are vacuous, and $\Tev^{\bP^r}_{g,n,d}$ simply equals the number of $L$-tableaux $\bL$ with parameters $(g,r,d)$.
The additional constraints on $\bL$ imposed in Theorem \ref{carlthm} are vacuous in the following three cases (cf. \cite[Theorem 1.3]{fl}):
\begin{enumerate}
    \item[(a)] $d\ge rg+r$, in which case $n-r-1$ exceeds the width $d-r$ of $\bL$.
    \item[(b)] $d=r+\frac{rg}{r+1}$ (equivalently, $n=r+2$, the smallest possible), in which case $B$ is empty.
    \item[(c)] $r=1$, in which case $n-r-1$ exceeds the size $r(n-r-2)$ of $B$.
\end{enumerate}
\end{remark}

\begin{example}
    When $g=0$, an $L$-tableau consists only of blue entries. In this case, there is a unique possible semi-standard filling $\bL$, which lies in $\cL^{r}_{0,n,d}$ because $n-r-1$ exceeds the width of $\bL$. Accordingly, we have $\Tev^{\bP^r}_{0,n,d}=1$. It will be convenient in some of our arguments later to assume that $g>0$.
\end{example}

\section{Operations on the red and blue tableaux}\label{sec:tableaux_operations}

In order to be able to apply the RSK correspondence, we next explain how the $L$-tableaux $\bL$ of Theorem \ref{carlthm} may be converted to a pair of tableaux $(P,Q)$ of the same shape, extending ideas of \cite{GenRSK}.

The following is noted in the proof of \cite[Lemma~3.4]{GenRSK}.

\begin{lemma}
\label{bwalemma}
Let $\bL$ be an $L$-tableau with parameters $(g,r,d)$ satisfying $d>g+r.$ Then, the leftmost $d-g-r$ columns of $\bL$ contain only blue integers.
\end{lemma}
\begin{proof}
As the red tableau of $\bL$ contains only integers among $1,2,\ldots,g$, and red entries strictly increase across rows from right to left, red integers may only appear in the rightmost $g$ columns of the $L$-tableau. As such, the leftmost $d-g-r$ columns contain only blue integers.
\end{proof}

\begin{definition}\label{cdef}
    Let $\text{SSYT}_\mathcal{C}(g,r,d)$ denote the set of SSYT's with:
    \begin{itemize}
    \item  size $(d-r)(r+1)-rg$,
    \item  entries among $\{1,2,\cdots, r+1\}$,
    \item  width at most $n-r-1=\frac{r+1}{r}d-g-r$,
    \item  at least $d-g-r$ columns of height $(r+1)$, and
    \item  no $(i,i+1)$-strip of length $n-r-1$ for any value $i\in \{1,2,\cdots, r\}.$
    \end{itemize}
If $d\le g+r$, then the fourth condition is empty. The subscript $\mathcal{C}$ refers to the set of constraints above.
\end{definition}

By Lemma \ref{bwalemma}, $\text{SSYT}_\mathcal{C}(g,r,d)$ is the set of all possible blue tableaux underlying $L$-tableaux satisfying the additional constraints of Theorem \ref{carlthm}. Our first step is to either truncate height $(r+1)$ columns from the blue tableau or add height $(r+1)$ columns, resembling the truncation in \cite[Lemma 3.4]{GenRSK}.

\begin{definition}\label{widthadjust}
    For $B\in\text{SSYT}_{\mathcal{C}}(g,r,d),$ we define the SSYT $\psi(B)$ to be the result of the following \emph{width adjustment}:
\begin{itemize}
\item If $d<g+r,$ then add $g+r-d$ columns of height $(r+1)$ to the left of $B$.
\item If $d>g+r,$ then remove the leftmost $d-g-r$ columns of $B$, necessarily of height $(r+1)$.
\item If $d=g+r,$ then $B$ is unchanged.
\end{itemize}

\end{definition}

By construction, the size of $\psi(B)$ is
\begin{equation*}
    (d-r)(r+1)-rg+(g+r-d)(r+1)=g
\end{equation*}
in all three cases.

% For $B\in \text{SSYT}_{\mathcal{C}}(g,r,d),$ we define a $180^\circ$-rotated SSYT $\psi(B)$ to be the result of the following \emph{width adjustment}:
% If $d<g+r,$ we add $g-d+r$ columns of height $(r+1)$ to the tableau. If $d>g+r,$ we truncate the rightmost $g-d+r$ columns of height $(r+1),$ this is always possible because of \cref{bwalemma}. If $d=g+r,$ the tableau is unchanged.

\begin{definition}\label{ACdef}
    Let $\text{SSYT}_{\mathcal{AC}}(g,r,d)$ denote the set of $180^\circ$-rotated SSYT's with 
    \begin{itemize}
    \item  size $g$, 
    \item  entries among $\{1,2,\cdots, r+1\}$,
    \item  width at most $(n-r-1)+(g-d+r)=\frac{r+1}{r}d-g-r+(g-d+r)=\frac{d}{r}$,
    \item  at least $g+r-d$ columns of height $(r+1)$,
    \item  No $(i,i+1)$-strip of length $(n-r-1)+(g-d+r)=\frac{d}{r}$ for any value $i\in \{1,2,\cdots, r\}.$
    \end{itemize}
If $d\ge g+r$, then the fourth condition is empty. (Note the change in sign from Definition \ref{cdef}.) The subscript $\mathcal{AC}$ refers to the set of ``adjusted'' constraints above.
\end{definition}

\begin{lemma}\label{psibij}
The map $\psi$ is a bijection from $\text{SSYT}_\mathcal C(g,r,d)$ to $\text{SSYT}_{\mathcal AC}(g,r,d)$ for all $g,r,d.$ 
\end{lemma}

\begin{proof}
    Fix a tableau $B\in \text{SSYT}_\mathcal C(g,r,d).$ Then, $\psi(B)\in \text{SSYT}_{\mathcal AC}(g,r,d)$. In particular, the fourth condition of Definition \ref{ACdef} is empty in the case $d\ge g+r$, and is automatically satisfied by $\psi(B)$ in the case $d<g+r$. The remaining conditions in Definition \ref{ACdef} are direct translations of their counterparts in Definition \ref{cdef}.

    Moreover, an inverse to $\psi$ is given as follows: if $d\ge g+r$, then add $d-g-r$ columns of height $(r+1)$ to a tableau $P\in \text{SSYT}_{\mathcal AC}(g,r,d),$ and if $d<g+r$, then remove $g+r-d$ columns of height $(r+1)$ to $P$, which is possible by the fourth condition of Definition \ref{ACdef}.
\end{proof}

\begin{example}\label{eg:add_column}
At left, a tableau $ B \in \text{SSYT}_\mathcal{C}(11,2,12) $, in particular the blue tableau of the $L$-tableau from Example \ref{fig:itsanltableau}. At right, the tableau $\psi(B) \in \text{SSYT}_\mathcal{AC}(11,2,12) $ formed by adding one additional height $3$ column to $B.$

\vspace{.55cm}
    \centering
    \label{fig:bla}
\begin{ytableau}
\textcolor{blue}{1}& \textcolor{blue}{1} & \textcolor{blue}{1} & \textcolor{blue}{2} & \textcolor{blue}{3}\\
 \textcolor{blue}{2}& \textcolor{blue}{2}\\
 \textcolor{blue}{3}
\end{ytableau}\hspace{2cm}
\begin{ytableau}
\textcolor{blue}{1}&\textcolor{blue}{1}& \textcolor{blue}{1} & \textcolor{blue}{1} & \textcolor{blue}{2} & \textcolor{blue}{3}\\
 \textcolor{blue}{2}& \textcolor{blue}{2}& \textcolor{blue}{2}\\
 \textcolor{blue}{3}& \textcolor{blue}{3}
\end{ytableau}

\vspace{.3cm}
\end{example}

\begin{definition}
Let $\text{TrSSYT}^{180^\circ}(g,r)$ be the set of all $180^\circ$-rotated and transposed SSYT's $R$ of content $(r^g)=(r,r,\ldots,r),$ and height at most $r+1$. 
\end{definition}

If $g>0$, then the number of non-empty rows of $R$ is either $r$ or $r+1$. We label the rows of $R$ by the integers $1,\ldots,r+1$ in such a way that the bottom row is non-empty and labelled $r+1$. Thus, if $g>0$, then rows $2,\ldots,r+1$ are necessarily non-empty, whereas row 1 may be empty. If $g=0$, then all rows $1,\ldots,r+1$ are taken to be empty.

\begin{definition}
Let $\text{SYT}(g,r)$ be the set of all SYT of size $g$ and height at most $r+1.$
\end{definition}

\begin{definition}\cite[Definition 3.9]{GenRSK}
    %Define the map $\phi:\text{SSYT}^{180^\circ}(g,r)\to \text{SYT}(g,r)$ as follows. 
    Given $R\in \text{SSYT}^{180^\circ}(g,r),$ define the SYT $\phi(R)$ with row set $1,\ldots,r+1$ as follows. For $i=1,\ldots,g$, on the $i$-th step, place a box labeled $i$ as far to the left as possible in the unique row among $1,\ldots,r+1$ that does not contain an $i$ in $R$. We call $\phi(R)\in \text{SYT}(g,r)$ the \emph{purple tableau} associated to $R$.
\end{definition}

By construction, $R$ and $\phi(R)$ are complementary in an $(r+1)\times g$ rectangle.

\begin{lemma}[\cite{GenRSK}, Lemma 3.11]\label{phibij}
The map $\phi$ is a bijection from $\text{SSYT}^{180^\circ}(g,r)$ to $\text{SYT}(g,r)$ for all $g,r.$ Moreover, for any $R\in \text{TrSSYT}(g,r),$ the shapes of $R$ and $\phi(R)$ are complementary in a $(r+1)\times g$ rectangle.
\end{lemma}
\begin{example}\label{eg:redtopurple} At left, a tableau $R\in \text{TrSSYT}^{180^\circ}(11,2),$ in particular the red tableau of the $L$-tableau from Figure \ref{fig:itsanltableau}. At right, the purple tableau $\phi(R)\in \text{SYT}(11,2).$

    \vspace{.55cm}
\centering
\begin{ytableau}
\none & \none & \none & \none &\textcolor{red}{8} & \textcolor{red}{6} & \textcolor{red}{5} & \textcolor{red}{3} & \textcolor{red}{2} \\
\none & \textcolor{red}{11} & \textcolor{red}{10} & \textcolor{red}{9} & \textcolor{red}{7}  & \textcolor{red}{6} & \textcolor{red}{4} & \textcolor{red}{3} & \textcolor{red}{1}\\
 \textcolor{red}{11} &\textcolor{red}{10} & \textcolor{red}{9} & \textcolor{red}{8} &\textcolor{red}{7}  & \textcolor{red}{5} & \textcolor{red}{4} & \textcolor{red}{2} & \textcolor{red}{1}
\end{ytableau}\hspace{2cm}
\begin{ytableau}
\textcolor{purple}{1} & \textcolor{purple}{4} & \textcolor{purple}{7}  & \textcolor{purple}{9} & \textcolor{purple}{10} & \textcolor{purple}{11}\\
\textcolor{purple}{2} &\textcolor{purple}{5} & \textcolor{purple}{8} \\
\textcolor{purple}{3} &\textcolor{purple}{6} \\
\end{ytableau}

\vspace{.3cm}

\end{example}

\begin{proposition}\label{prop:convert_to_pair_of_same_size}
    Let $\cL^{r}_{g,n,d}$ be as in Definition \ref{def:calL}, and let $\cT_{g,n,d}^r$ denote the set of pairs $(P,Q)$ of the same shape with $P\in \text{SSYT}_\mathcal{AC}(g,r,d)$ and $Q$ a SYT.
    Then, the map
    \begin{equation*}
        \cL^r_{g,n,d}\to \cT^r_{g,n,d},
    \end{equation*}
   defined by $\bL\mapsto (\psi(B),\phi(R))$, is a bijection.
\end{proposition}

Note that if $(P,Q)\in \cT_{g,n,d}^r$, then it is automatic that $Q\in \SYT(g,r)$ from the requirement that $P,Q$ have the same shape.

\begin{proof}
The map given by $\bL \mapsto (B,R),$ from $\cL^{r}_{g,n,d}$ to the set of pairs $(B,R)$ complementary in an $(r+1)\times (d-r)$ rectangle with $B\in \text{SSYT}_{\mathcal C}(g,r,d)$ and $R\in \text{TrSSYT}^{180^\circ}(g,r),$ is a bijection. By Proposition \ref{psibij} and lemma \ref{phibij}, the maps
\begin{align*}
    \psi:\text{SSYT}_{\mathcal C}(g,r,d)&\rightarrow \text{SSYT}_{\mathcal AC}(g,r,d)\\
    \phi:\text{TrSSYT}^{180^\circ}(g,r)&\rightarrow \text{SYT}(g,r)
\end{align*}
are both bijections. By construction of $\psi$, the tableaux $\psi(B)$ and $R$ are complementary in an $(r+1)\times g$ rectangle. By construction of $\phi$, the tableaux $\phi(R)$ and $R$ are also complementary in an $(r+1)\times g$ rectangle, so $\psi(B)$ and $\phi(R)$ have the same shape. The conclusion follows.
\end{proof}

\begin{remark}
The association $\bL\mapsto (\psi(B),\phi(R))$ also defines a bijection, from the set of all $L$-tableaux $\bL$ as in Definition \ref{Ltableaux}, to the set of pairs $(P,Q)$ where:
\begin{itemize}
    \item $P$ is an SSYT and $Q$ is an SYT, of the same shape $\lambda$ with $|\lambda|=g$,
    \item $P$ is filled with entries among $1,\ldots,r+1$, and
    \item $\lambda$ has height at most $r+1$, and at least $g+r-d$ columns of height exactly $r+1$.
\end{itemize}
\end{remark}

\section{RSK}\label{sec:rsk}

\begin{definition}\label{def:word}
   Let $g,r$ be integers. An \emph{$(r+1)$-ary word of length $g$} is a sequence $x_1,\ldots,x_g$, where $x_1,\ldots,x_g\in\{1,\ldots,r+1\}$. 
\end{definition}

In \cite{GenRSK}, it is observed that, if $d\ge g+r$, then composing the map $\bL\mapsto (\psi(B),\phi(R))$ with the RSK correspondence gives a bijection from the \emph{entire} set of $L$-tableaux with parameters $(g,r,d)$ to the set of $(r+1)$-ary words of length $g$. If in addition $d\ge rg+r$, then $\cL^r_{g,n,d}$ is the entire set of $L$-tableaux (Remark \ref{rem:conditions_vacuous}(a)), which recovers the fact that $\Tev^{\bP^r}_{g,n,d}=(r+1)^g$ (Theorem \ref{thm:large_deg}).

With Proposition \ref{prop:convert_to_pair_of_same_size} in hand, we therefore wish to compute the image under RSK of $\cT^{r}_{g,n,d}$. The subset of $(r+1)$-ary words obtained will in turn be in bijection with curves enumerated by $\Tev^{\bP^r}_{g,n,d}$, for \emph{all} tuples $(g,r,d,n)$ satisfying \eqref{eq:dim_constraint}.

\begin{example}
Applying the RSK algorithm to the $3$-ary sequence $3,2,1,3,2,1,2,1,1,2,3$ yields the pair of tableaux $(\psi(B),\phi(R))$ (as in Examples \ref{eg:add_column} and \ref{eg:redtopurple}).

 \vspace{.55cm}
\centering
\begin{ytableau}
\textcolor{blue}{1}&\textcolor{blue}{1}& \textcolor{blue}{1} & \textcolor{blue}{1} & \textcolor{blue}{2} & \textcolor{blue}{3}\\
 \textcolor{blue}{2}& \textcolor{blue}{2}& \textcolor{blue}{2}\\
 \textcolor{blue}{3}& \textcolor{blue}{3}
\end{ytableau}\hspace{2cm}
\begin{ytableau}
\textcolor{purple}{1} & \textcolor{purple}{4} & \textcolor{purple}{7}  & \textcolor{purple}{9} & \textcolor{purple}{10} & \textcolor{purple}{11}\\
\textcolor{purple}{2} &\textcolor{purple}{5} & \textcolor{purple}{8} \\
\textcolor{purple}{3} &\textcolor{purple}{6} \\
\end{ytableau}

\vspace{.3cm}

\end{example}

\begin{definition}\label{def:decreasing}
A \emph{nondecreasing} subsequence of an $(r+1)$-ary word $x_1,\ldots, x_g$ is a word $x_{j_1},\ldots ,x_{j_m}$ with \[x_{j_1}\leq x_{j_2}\leq \dots \leq x_{j_m}\]
and \[1\le j_1<j_2<\dots <j_m\le g.\]

A \emph{decreasing} subsequence is defined similarly. We say that a collection of nondecreasing (or decreasing) subsequences of a word is \emph{disjoint} if their index sets $\{j_1,\ldots,j_m\}$ are pairwise disjoint.
\end{definition}

\begin{definition}\label{def:(i,i+1)_subseq}
    An \emph{$(i,i+1)$-subsequence} of a word $x_1,\ldots,x_g$ is a nondecreasing subsequence, all of whose letters $x_{j_1}\le \cdots\le x_{j_m}$ are equal to $i$ or $i+1$.
\end{definition}

We next recall the definition of the \emph{reading word} of a tableau and the elementary Knuth transformations:

\begin{definition}\label{readingdef}
Let P be a SSYT. The reading word of $P,$ denoted
$\text{reading}(P),$ is the sequence of entries of $T$ obtained by concatenating the rows of $T$ from bottom to top.

\end{definition}

\begin{example}  The tableau $\psi(B) \in \text{SSYT}_\mathcal{AC}(6,3,10) $ from Example \ref{eg:add_column} has reading word $3,3,2,2,2,1,1,1,1,2,3.$

\vspace{.55cm}
    \centering
    \label{fig:bla}
\begin{ytableau}
\textcolor{blue}{1}& \textcolor{blue}{1}& \textcolor{blue}{1} & \textcolor{blue}{1} & \textcolor{blue}{2} & \textcolor{blue}{3}\\
 \textcolor{blue}{2} & \textcolor{blue}{2} & \textcolor{blue}{2}\\
 \textcolor{blue}{3} & \textcolor{blue}{3}
\end{ytableau}

\vspace{.3cm}
\end{example}

\begin{definition}
An \textit{elementary Knuth transformation} of a word applies one of the following transformations or their inverses to three consecutive letters in a word:
\[\cdots acb\cdots\rightarrow \cdots cab \cdots \hspace{.3cm }(a\leq b <c)\]
\[\cdots bac\cdots\rightarrow \cdots bca \cdots \hspace{.3cm }(a< b \leq c).\]
Two words $\omega$ and $\nu$ are \textit{Knuth equivalent} if one of them can be obtained from the other via a sequence of elementary Knuth transformations.
\end{definition}

\begin{theorem}\cite[\S 2.1]{fultonYT}
 Any word is Knuth equivalent to the reading word of
its insertion tableau. Words are Knuth equivalent if and only if their insertion tableaux coincide. 
\end{theorem}

% In the following lemma, let $T$ of shape $\lambda$ denote the insertion tableaux resulting from applying the RSK algorithm to a $(r+1)$-ary word  $\omega.$ The following is evident from Definition \ref{ii+1stripdef} and Definition \ref{readingdef}.

\begin{lemma}\label{lemma:striptoreading}
Let $P$ be an SSYT of shape $\lambda$. Then, $P$ has an $(i,i+1)$-strip of length $\lambda_1$ if and only if $\text{reading}(P)$ has an $(i,i+1)$-subsequence of length $\lambda_1.$
\end{lemma}

Recall that $\lambda_1$ is by definition the length of the largest part of $\lambda$.

\begin{proof}
    Take a tableau $P$ with an $(i,i+1)$-strip of length $\lambda_1.$ By the conditions in  Definition \ref{ii+1stripdef}, the entries in the strip will correspond directly to an $(i,i+1)$ sequence in $\text{reading}(P).$
    
    For the reverse direction, take a tableau $P$ such that $\text{reading}(P)$ has an $(i,i+1)$-subsequence of length $\lambda_1$, with $k$ entries $i$ followed by $\lambda_1-k$ entries $i+1.$ As the entries of this $(i,i+1)$-subsequence are weakly increasing, it follows from Definition \ref{readingdef} that the corresponding entries in $P$ will consist of $k$ entries $i$ which lie to the left of and weakly below $\lambda_1-k$ entries $(i+1).$ Since $T$ has width $\lambda_1,$ there must be an entry $i$ in the first $k$ columns of $P$ and an entry $i+1$ in the final $\lambda_1-k$ columns of $P,$ showing that $P$ indeed has an $(i,i+1)$-strip of length $\lambda_1.$  
\end{proof}

\begin{lemma}\label{lemma:striptosequence}
Let $\omega$ be a word. Let $P$ be the insertion tableau obtained from applying RSK to $\omega$, and let $\lambda$ be the shape of $P$.
Then, $P$ has an $(i,i+1)$-strip of length $\lambda_1$ if and only if $\omega$ has an $(i,i+1)$-subsequence of length $\lambda_1.$
\end{lemma}

\begin{proof}
We begin by noting that the property of having an $(i,i+1)$-subsequence of length $\ell$ is preserved by Knuth transformations. In fact, the only elementary Knuth transformations that affect the relative position of entries $(i)$ and $(i+1)$ are given by the following  and their inverses: \[\cdots (i)(i+1)(i)\cdots\rightarrow \cdots (i+1)(i)(i)\cdots\] \[\cdots(i+1)(i)(i+1)\cdots\rightarrow \cdots(i+1)(i+1)(i)\cdots.\] In both cases, although the relative order of the entries $(i)$ and $(i+1)$ changes, the property of having a $(i,i+1)$-subsequence of length $\ell$ is preserved. 

As stated in Lemma \ref{lemma:striptoreading}, $P$ has an $(i,i+1)$-strip of length $\lambda_1$ if and only if $\text{reading}(P)$ has an $(i,i+1)$-subsequence of length $\lambda_1.$ As $\text{reading}(P)$ is Knuth equivalent to $\omega,$ we conclude that $P$ has an $(i,i+1)$-strip if and only if $\omega$ has an $(i,i+1)$-subsequence of length $\lambda_1.$
\end{proof}

We next recall the statement of Greene's Theorem.

\begin{theorem}\cite[Theorem 4.8.10]{sagan}
\label{Greenestheorem}
Let $\omega$ be a word. Let $\lambda$ be the shape of the pair of tableaux obtained by applying RSK to $\omega$, and let $\lambda'$ be its conjugate.

Let $I_k(\omega)$ be the maximum number of elements in a collection of $k$ disjoint nondecreasing subsequences of $\omega$. Similarly, let $D_{k}(\omega)$ be the maximum number of elements in a union of $k$ disjoint decreasing subsequences of $\omega$.
Then, the following hold:
\begin{align*}
I_k(\omega)&=\lambda_1 + \cdots + \lambda_k,\\
D_{k}(\omega)&=\lambda'_1+\cdots \lambda'_{k}.
\end{align*}
\end{theorem}

\begin{proposition}\label{prop:calTtosequences}

The pairs of tableaux in $\cT_{g,n,d}^r,$ as defined in Proposition \ref{prop:convert_to_pair_of_same_size}, are in bijection with $(r+1)$-ary words of length $g$ that have:

\begin{enumerate}
\item[(i)] a collection of at least $g-d+r$ disjoint decreasing subsequences of length $r+1$,
\item[(ii)] no nondecreasing subsequence of length greater than $\frac{d}{r}$, and
\item[(iii)] no $(i,i+1)$-subsequence of length $\frac{d}{r}.$ 
\end{enumerate}

\end{proposition}

\begin{proof}
Consider a pair of tableaux $(P,Q)\in \cT_{g,n,d}^r$ of shape $\lambda\vdash g$ which result from applying RSK to an $(r+1)$-ary word $\omega,$ with $P$ the insertion tableau and $Q$ the recording tableau. 

From Definition \ref{ACdef}, the tableaux $P$ and $Q$ must have at least $g+r-d$ columns of height $(r+1)$ and width at most $\frac{d}{r}.$ Applying Theorem \ref{Greenestheorem}, we obtain that the word $\omega$ has at least $g-d+r$ disjoint decreasing subsequences of maximal length $r+1,$ and no nondecreasing subsequence of length greater than $\frac{d}{r}.$ 

In addition, Definition \ref{ACdef} requires that $P$ contain no $(i,i+1)$-strip of length $\frac{d}{r}.$ 
If the width of $P$ is less than $\frac{d}{r},$ then this condition is satisfied automatically, and the absence of a nondecreasing subsequence of length $\frac{d}{r}$ already implies the desired constraint on $(i,i+1)$-subsequences in $\omega.$ 
If the width of $P$ is exactly $\frac{d}{r},$ then Lemma \ref{lemma:striptoreading} shows that $\omega$ has no $(i,i+1)$-subsequence of length $\frac{d}{r}.$

Thus, every element of $\cT_{g,n,d}^r$ corresponds to a word $\omega$ satisfying the three listed condition. The reverse direction follows similarly from a direct application of Theorem \ref{Greenestheorem} and Lemma \ref{lemma:striptosequence}.
\end{proof}

Recall the sets $\cL_{g,n,d}^r$ from Definition \ref{def:calL} and $\cT_{g,n,d}^r$ from Proposition \ref{prop:convert_to_pair_of_same_size}. By Proposition \ref{prop:convert_to_pair_of_same_size}, there is a bijection between $\cL_{g,n,d}^r$ and $\cT_{g,n,d}^r.$

\begin{proof}[Proof of Theorem~\ref{thm:main_intro}]
From Theorem \ref{carlthm}, the count $\text{Tev}_{g,n,d}^{\mathbb{P}^r}$ is equal to the size of the set of $L$-tableaux $\cL_{g,n,d}^r.$. Then Proposition \ref{prop:convert_to_pair_of_same_size} gives a bijection between $\cL_{g,n,d}^r$ and $\cT_{g,n,d}^r,$ and Proposition \ref{prop:calTtosequences} further gives a bijection between $\cT_{g,n,d}^r$ and words satisfying the conditions of Theorem~\ref{thm:main_intro}.
\end{proof}

\section{Remarks on words}\label{sec:remarks}

We conclude with some remarks on the set of words enumerated by Theorem \ref{thm:main_intro}. As pointed out in \S\ref{sec:intro}, all three of the conditions (i)-(iii) are vacuous when $d\ge rg+r$, so we recover the fact that $\Tev^{\bP^r}_{g,n,d}=(r+1)^g$ in this range. The bound $d\ge rg+r$ is sharp: suppose instead that $d<rg+r$, in which case the requirement that $d|r$ forces $d\le rg$. Then, the sequences enumerated in Theorem \ref{thm:main_intro} do not comprise all words $x_1,\ldots,x_g$, as constant sequences always fail to satisfy (iii). 

If $d=rg$, then conditions (i) and (ii) are again vacuous, but condition (iii) disallows words which consist entirely of an $(i,i+1)$ subsequence. In total, there are
\begin{equation*}
    (g-1)r+(r+1)=n
\end{equation*}
such words, where the first term counts words that contain each of the letters $i,i+1$ at least once (for some $i\in\{1,\ldots,g-1\}$), and the second term counts constant sequences. It follows that $\Tev^{\bP^r}_{g,n,d=rg}=(r+1)^g-n$, see also \cite[Example 1.6]{tev_Pr}. The $n$ disallowed words correspond to stable maps (in the sense of Gromov-Witten theory) which contract a general curve to one of the $n$ points $x_i$ among $x_1,\ldots,x_n\in\bP^n$, and map rational tails to lines joining $x_i$ and the other $n-1$ points.

More generally, Theorem \ref{thm:main_intro} also makes manifest the fact that
 \begin{equation*}
     \Tev^{\bP^r}_{g,n,d}\le \Tev^{\bP^r}_{g,n+r+1,d+r},
 \end{equation*}
see also \cite[\S 5.2]{tev_Pr}. Indeed, a word satisfying (i)-(iii) still does so upon replacing $n,d$ with $n+r+1,d+r$. Geometrically, this corresponds roughly to the following phenomenon. Consider the maps $f:C\to\bP^r$ enumerated by $\Tev^{\bP^r}_{g,n+r+1,d+r}$, which satisfy $n+r+1$ conditions of the form $f(p_i)=x_i$. Suppose that the last $r+1$ of the points $p_1,\ldots,p_{n+r+1}\in C$ are moved from general into special position, becoming equal to a single point $p$. Then, the flat limits of the maps $f$ come in the following two types.
\begin{enumerate}
\item $\Tev^{\bP^r}_{g,n,d}$ \emph{quasimaps} $f':C\to\bP^r$ of degree $d+r$ with an order $r$ base-point at $p$, and satisfying $f'(p_i)=x_i$ for the remaining $n$ points. After twisting down the base-points, such $f'$ may be interpreted simply as maps of degree $d$.
\item $\Tev^{\bP^r}_{g,n+r+1,d+r}-\Tev^{\bP^r}_{g,n,d}\ge 0$ further degenerate contributions.
\end{enumerate}
In general, the degenerate contributions of the second type do not admit as simple of a geometric description.

When $r=1$, the description of the words enumerated in Theorem \ref{thm:main_intro} simplifies: (iii) implies (ii), and (i) implies them both. Indeed, a nonincreasing subsequence when $r=1$ is simply a $(1,2)$-subsequence, hence (iii) implies (ii). Now, if a binary word $\omega$ of length $g$ contained a $(1,2)$-subsequence $\omega'$ of length $d$, then a decreasing subsequence of $\omega$ of length 2 could contain at most one letter coming from $\omega'$. Thus, if (i) holds for $\omega$, then the total length of $\omega$ would be at least $d+(g+1-d)=g+1$, a contradiction.

Now, let $\omega$ is a binary word of length consisting of $a$ 1's and $b$ 2's, where $a+b=g$. Then, condition (i) is in turn equivalent to the following. First, we must have $a,b\ge g+1-d$. Then, reading from left to right, the number of 1's may never exceed the number of 2's by more than $d-1-b$. That is, if $\omega$ is interpreted as a lattice path, where 1's correspond to rightward steps and 2's correspond to upward steps, $\Tev^{\bP^1}_{g,n,d}$ counts the number of up-right paths to from $(0,0)$ to some point $(a,b)$, where $a,b$ are as above, not traveling below the line $y=x-(d-1-b)$.

In the case $d=\frac{g}{2}+1$, one must have $a=b=\frac{g}{2}$, and the familiar interpretation of the Catalan numbers in terms of lattice paths is recovered. Indeed, the fact that 
\begin{equation*}
    \Tev^{\bP^1}_{g,3,g/2+1}=C_{g/2}=\frac{1}{(g/2)+1}\binom{g}{g/2}
\end{equation*}
goes back to the 19th century calculation of Castelnuovo \cite{castelnuovo}.

\bibliographystyle{alpha} 
\bibliography{Tev_RSK_v3.bib}

\end{document}